\documentclass{article}
\usepackage{graphicx}
\graphicspath{{./pictures/}}
\usepackage{amsmath,amsfonts,amsthm}
\allowdisplaybreaks

\usepackage{url}

\usepackage{amsmath,amssymb,amsthm}
\usepackage[colorlinks=true]{hyperref}
\hypersetup{urlcolor=blue, citecolor=red}

\allowdisplaybreaks

\newtheorem{theorem}{Theorem}
\theoremstyle{definition}

\title{Lagrangian dynamics \\
by nonlocal constants of motion}

\date{\emph{June 3, 2019}}

\begin{document}
\maketitle

\centerline{\scshape Gianluca Gorni}
\medskip
{\footnotesize
 \centerline{Universit\`a di Udine}
   \centerline{Dipartimento di Scienze Matematiche, Informatiche e Fisiche}
   \centerline{via delle Scienze~208, 33100 Udine, Italy}
   \centerline{\tt{gianluca.gorni@uniud.it}}}
 

\medskip

\centerline{\scshape Gaetano  Zampieri}
\medskip
{\footnotesize
 \centerline{Universit\`a di Verona}
   \centerline{Dipartimento di Informatica}
   \centerline{strada Le Grazie 15, 37134 Verona, Italy}
   \centerline{\tt{gaetano.zampieri@univr.it}}}


\begin{abstract}
A simple general theorem is used as a tool that generates nonlocal constants of motion for Lagrangian systems. We review some cases where the constants that we find are useful in the study of the systems: the homogeneous potentials of degree~$-2$, the mechanical systems with viscous fluid resistance and the conservative and dissipative Maxwell-Bloch equations of laser dynamics. We also prove a new result on explosion in the past for mechanical system with hydraulic (quadratic) fluid resistance and bounded potential.
\end{abstract}

\section{Introduction}
\label{introduction}

Consider the finite-dimensional \emph{variational} Euler-Lagrange equation 
 \begin{equation}\label{Euler}
  \frac{d}{dt}\partial_{\dot q}L\bigl(t,q(t),\dot q(t)\bigr)-
  \partial_qL\bigl(t,q(t),\dot q(t)\bigr)
  =0\,,
\end{equation}
where the Lagrangian  $L(t,q,\dot q)$ is a smooth function, with $t\in \mathbb{R}$, $q,\dot q\in\mathbb{R}^n$. We use the notation $\partial_{q}$ and $\partial_{\dot q}$ for the partial derivative operators with respect to the vector~$q$ and~$\dot q$ respectively, $\lvert x\rvert$ and $x\cdot y$~for the Euclidean norm and scalar product of vectors~$x,y\in\mathbb{R}^n$.

A~first integral is a smooth function of the form
\begin{equation}\label{trueFirstIntegral}
  N(t,q,\dot q),\qquad t\in \mathbb{R},
  \quad q,\dot q\in\mathbb{R}^n,
\end{equation}
that is constant along all solutions to Euler-Lagrange equation. The celebrated Noether's theorem establishes a connection between first integrals and certain invariance properties of the Lagrangian function~$L$.

A previous work of ours~\cite{GorniZampieri} revisited Noether's Theorem from different points of view, including asynchronous perturbations (or ``time change'') and boundary  terms, this last being a nomenclature recommended by Leach~\cite{Leach}. In the present paper we focus on the extension we obtained to constants of motion of the more general form
\begin{equation}\label{genericNonlocaConstantOfMotion}
  N\bigl(t,q(t),\dot q(t)\bigr)+
  \int_{t_0}^t M\bigl(s,q(s),\dot q(s)\bigr)ds\,,
\end{equation}
which we call \emph{nonlocal}, because its value at a time~$t$ depends not only on the value of position and velocity at time~$t$, but also on the past history of the motion.

In later works (see, e.g.,~\cite{GZKilling}) we extended the results to the \emph{nonvariational} case, where an extra term~$Q$ appears on the right-hand side of the differential equation, as in formula~\eqref{Lagrange} below. For such systems the motions are not stationary points of the action functional associated with the Lagrangian~$L$, in the sense of the calculus of variations.

The basic, very simple result on nonlocal constants of motion in that paper~\cite{GZKilling} can be reformulated in the following self-contained way, which is all that is needed for the sequel:

\begin{theorem}\label{mainTheorem}
Let $t\mapsto q(t)$ be a solution to the Lagrange equation
\begin{equation}
\label{Lagrange}
  \frac{d}{dt}\partial_{\dot q}L\bigl(t,q(t),\dot q(t)\bigr)-
  \partial_qL\bigl(t,q(t),\dot q(t)\bigr)
  = Q\bigl(t,q(t),\dot q(t)\bigr) 
\end{equation}
for smooth $L(t,q,\dot q)$, $Q(t,q,\dot q)$, with $t\in \mathbb{R}$, $q,\dot q\in\mathbb{R}^n$, and let $q_\lambda(t)$, $\lambda\in \mathbb{R}$, be a smooth family of  perturbed motions, such that $q_0(t)\equiv q(t)$. Then the following function is constant:
\begin{multline}\label{veryGeneralConstantAlongMotion}
  t\mapsto
    \boxed{\partial_{\dot q}
  L\bigl(t,q(t),\dot q(t)\bigr)\cdot
  \partial_\lambda q_\lambda(t)
  \big|_{\lambda=0}-
  \int_{t_0}^t\biggl(
  \frac{\partial}{\partial\lambda}
  L\bigl(s,q_\lambda(s),\dot q_\lambda(s)\bigr)
  \big|_{\lambda=0}+{}}\\
  \boxed{{}+
  Q\bigl(s,q(s),\dot q(s)\bigr)\cdot \partial_\lambda q_\lambda(s)
  \big|_{\lambda=0}\biggr)ds}\,.
\end{multline}
\end{theorem}
\begin{proof}
Take $\frac{d}{dt}$ of~\eqref{veryGeneralConstantAlongMotion}, and use~\eqref{Lagrange} and $\frac{d}{dt}\partial_\lambda q_\lambda(t)
 =\partial_\lambda\dot q_\lambda(t)$ at $\lambda=0$.
\end{proof}

The constant \eqref{veryGeneralConstantAlongMotion} is often trivial or of no apparent practical value, but there are cases when it is interesting and useful. In the rest of this paper we will review some applications in the variational case
\begin{itemize}
\item potentials with simple symmetries in Section~\ref{symmetrySection} as basic motivation,
\item homogeneous potentials of degree~$-2$ in Section~\ref{hompotentialsdegree-2}, taken from~\cite{GorniZampieri},
\item viscous fluid resistance in Section~\ref{viscousfluidresistance}, taken from~\cite{GZviscous},
\end{itemize}
and two in the nonvariational case:
\begin{itemize}
\item hydraulic fluid resistance in Section~\ref{hydraulicfluidresistance},
\item the Maxwell-Bloch equations for laser dynamics in Section~\ref{MBequations}, taken from~\cite{GZMBcons} in the conservative case and from~\cite{GZKilling} and~\cite{laserdiss} in the dissipative case.
\end{itemize}
The result in Section~\ref{hydraulicfluidresistance} is actually new: for a particle in $\mathbb{R}^n$ under quadratic fluid resistance and a bounded, nonnegative potential energy, we prove the explosion in the past in finite time of all solutions with initial kinetic energy greater than the upper bound of the potential energy.

\section{Lagrangians with simple symmetries}
\label{symmetrySection}

The perturbed motions $q_\lambda(t)$ of Theorem~\ref{mainTheorem} were originally inspired by the mechanism that Noe\-ther's theorem uses to deduce conservation laws for variational Lagrangian systems (for which $Q\equiv0$) under certain symmetry conditions on~$L$. A~simple example is a particle of mass $m$ in the plane that is driven by a central force field 
\begin{equation}
  L(t,q,\dot q):=\frac{1}{2}m \lvert\dot q\rvert^2-
  U\bigl(t,\lvert q\rvert\bigr),
  \quad  q=(q_1,q_2)\in \mathbb{R}^2.
\end{equation} 
To exploit the rotational symmetry of~$L$ it is natural to take the \emph{rotation family}
\begin{equation}
  q_\lambda(t):=\begin{pmatrix}
  \cos\lambda & -\sin\lambda\\
  \sin\lambda & \cos\lambda
  \end{pmatrix} q(t),\qquad
  \partial_\lambda q_\lambda(t)
  \big|_{\lambda=0}=\bigl(-q_2(t), q_1(t)\bigr)\,.
\end{equation}
It is clear that $L(t,q_\lambda(t),\dot q_\lambda(t))$ does not depend on~$\lambda$. Formula~\eqref{veryGeneralConstantAlongMotion} reduces to a simple version of Noether's theorem and gives the \emph{angular momentum} as constant of motion:
\begin{equation*}
  \partial_{\dot q}
  L\cdot\partial_\lambda q_\lambda
  \big|_{\lambda=0}=m\dot q\cdot (-q_2, q_1)=
  m(q_1\dot q_2-q_2\dot q_1).
\end{equation*}

A simple, somewhat less conventional use of the theorem is the following. For time independent $L(t,q,\dot q)=\mathcal{L}(q,\dot q)$, $Q\equiv 0$, and the \emph{time-shift} family $q_{\lambda}(t)=q(t+\lambda)$ we have
\begin{equation*}
  \partial_{\lambda}L\bigl(t,q_{\lambda}(t),
  \dot q_{\lambda}(t)\bigr)
  \big|_{\lambda=0}=\partial_{q}\mathcal{L}\cdot \dot q(t)+
  \partial_{\dot q}\mathcal{L}\cdot \ddot q(t).
\end{equation*}
The constant of motion is
\begin{align*}
  \partial_{\dot q}\mathcal{L}\cdot \dot q(t)-{}&
  \int_{t_0}^t \frac{d}{ds}\mathcal{L}\bigl(q(s),\dot q(s)\bigr)ds=\\  
  &=\partial_{\dot q}\mathcal{L}\bigl(q(t),\dot q(t)\bigr)\cdot
  \dot q(t)-\mathcal{L}\bigl(q(t),\dot q(t)\bigr)+
  \mathcal{L}\bigl(q(t_0),\dot q(t_0)\bigr)=\\
  &=E\bigl(q(t),\dot q(t)\bigr)+
  \mathcal{L}\bigl(q(t_0),\dot q(t_0)\bigr),
\end{align*}
which coincides with the \emph{energy}
\begin{equation}
 E(q,\dot q)=
  \partial_{\dot q}\mathcal{L}(q,\dot q)\cdot \dot q-\mathcal{L}(q,\dot q)
\end{equation}
up to the additive constant $\mathcal{L}(q(t_0),\dot q(t_0))$.
For instance, when the Lagrangian is $\mathcal{L}(q,\dot q):=\frac{1}{2}m \lvert\dot q\rvert^2- U(q)$, the conserved energy takes the classical form of kinetic plus potential energies: $E(q,\dot q)=\frac{1}{2}m\lvert \dot q\rvert^2+U(q)$.

\section{Homogeneous potentials of degree $-2$}
\label{hompotentialsdegree-2}

In this section we are going to review a result in our previous work \cite{GorniZampieri}, Section~9. Consider the variational mechanical system of a point moving in a potential field:
\begin{equation}\label{hom}
  L(t,q,\dot q):=
  \frac{1}{2}m\lvert\dot q\rvert^2-U(q),\qquad
  Q\equiv 0,\qquad
  m\ddot q=-\nabla U(q),
\end{equation}
and assume that the potential $U$ is positively \emph{homogeneous} of degree~$-2$: 
\begin{equation}
  U(sq)=s^{-2}U(q), \qquad\forall s>0.
\end{equation}

Two notable examples are the \emph{central} potential case
\begin{equation}
  U(q)= -k/|q|^2, \qquad q\in\mathbb{R}^n\setminus\{0\},
\end{equation}
and \emph{Ca\-lo\-ge\-ro's}  potential
\begin{equation*}
  U(q_1,\dots,q_n)=g^2\sum_{1\le j<k\le n}(q_j-q_k)^{-2},
\end{equation*}
for $q_j\in\mathbb{R}$, $q_j\ne q_k$ when $j\ne k$, see Calogero's paper \cite{Calogero}.

All these systems enjoy a remarkable symmetry: if $q(t)$ is solution to the last of~\eqref{hom}, then
\begin{equation*}
  q_\lambda(t)=e^\lambda q\bigl(e^{-2\lambda}t\bigr), 
\end{equation*}
is a solution too. Theorem~\ref{mainTheorem} associates to this family $q_\lambda$ the following constant of motion
\begin{align*}
  K={}&\partial_{\dot q}L\cdot\partial_\lambda q_\lambda(t)
  \big|_{\lambda=0}
  -\int\frac{\partial}{\partial\lambda}
  L\bigl(t,q_\lambda(t),\dot q_\lambda(t)\bigr)
  \big|_{\lambda=0}dt=\\
  ={}&m\dot q\cdot (q-2t\dot q)-2tU=
  m\dot q\cdot (q-2t\dot q)+2tL. 
\end{align*}
This time-dependent local constant of motion can be rewritten in terms of the energy~$E$, which is constant too:
\begin{equation*}
  K=m\dot q\cdot q-2t E,\qquad 
  E:=\frac{1}{2}m\lvert\dot q\rvert^2+U(q).
\end{equation*}
Take the antiderivative in time of $0=mq\cdot\dot q-2tE-K$ and obtain one more  constant of motion
\begin{equation*}
  K_1=\frac{1}{2}m\lvert q\rvert^2-t^2E-tK\,.
\end{equation*}
We can  solve for~$\lvert q\rvert$:
\begin{equation}
   \bigl|q(t)\bigr|=\sqrt{\frac{2}{m}}\sqrt{t^2E+tK+K_1}\,.
\end{equation}
This formula gives the explicit \emph{time-dependence of the distance from the origin}, even though we don't know the shape of the orbit.

\section{Viscous fluid resistance}
\label{viscousfluidresistance}

This Section reviews the main result of our paper \cite{GZviscous}.
Consider a particle under a \emph{bounded from below} potential $U\colon\mathbb{R}^n\to\mathbb{R}$ and \emph{viscous (linear) fluid resistance}:
\begin{equation}\label{dissipativeequation}
  m\ddot q=-k\dot q-\nabla U(q),
  \qquad U\ge 0,\quad q\in\mathbb{R}^n,
\end{equation}
where $m$ and $k>0$ are parameters. The mechanical energy 
\begin{equation}\label{mechanicalenergy}
  \frac{1}{2}m\lvert\dot q\rvert^2+U(q)
\end{equation}
decreases along solutions $q(t)$ and $\dot q(t)$ is bounded in the future:
\begin{equation*}
  \frac{1}{2}m \bigl|\dot q(t)\bigr|^2\le
 \frac{1}{2}m \bigl|\dot  
  q(t_0)\bigr|^2+U\bigl(q(t_0)\bigr),
  \qquad t\ge t_0.
\end{equation*}
So $q(t)$ is bounded for bounded $t$ and we get \emph{global existence  in the future}. What about the past? Equation~\eqref{dissipativeequation} can be put into Lagrangian form \eqref{Lagrange} with
\begin{equation*}
   L(t,q,\dot q):=e^{kt/m}\Bigl(
   \frac{1}{2}m \lvert\dot q\rvert^2-U(q)
   \Bigr),\qquad Q\equiv 0.
\end{equation*}
Incidentally, a study of Noether symmetries and conservation laws for this Lagrangian function has been made by Leone and Gourieux~\cite{Leone}.

Let us apply Theorem~\ref{mainTheorem} with the family $q_\lambda(t):=q(t+\lambda e^{kt/m})$. Then computing the nonlocal constant of motion~\eqref{veryGeneralConstantAlongMotion} and integrating by parts we have a simple formula for the constant of motion:
\begin{equation*}
  e^{2kt/m}\Bigl(m\lvert\dot q(t)\rvert^2
  +2U\bigl(q(t)\bigr)\Bigr)
  -\frac{4k}{m}\int^{t}_{t_0} e^{2ks/m}
  U\bigl(q(s)\bigr)ds.
\end{equation*}
Since $U\ge 0$, the integral term increases with~$t$, forcing the  remaining part 
$$e^{2kt/m}\Bigl(m\lvert\dot q(t)\rvert^2+2U\bigl(q(t)\bigr)\Bigr)$$ to be increasing too. We deduce the inequalities
\begin{align}
  me^{2kt/m}\lvert\dot q(t)\rvert^2\le{}&
  e^{2kt/m}
  \Bigl(\lvert\dot q(t)\rvert^2+2U\bigl(q(t)\bigr)
  \Bigr)\le\\
  \le{}&e^{2kt_0/m}
  \Bigl(\lvert\dot q(t_0)\rvert^2+2U\bigl(q(t_0)\bigr)
  \Bigr)\qquad \forall t\le t_0.
\end{align}
In a bounded interval $(t_1,t_0]$ the velocity $\dot q(t)$ is bounded, and therefore $q(t)$ is~too. This proves \emph{global existence of solutions also in the past}.

\section{Explosion in the past for hydraulic fluid resistance}
\label{hydraulicfluidresistance}

We are going to see a new result. Consider the equation for \emph{hydraulic resistance in a bounded potential field}:
\begin{equation}\label{quadraticResistanceEquation}
  m\ddot q=-k\lvert \dot q\rvert\dot q
  -\nabla U\bigl(q(t)\bigr),
  \quad q\in\mathbb{R}^n,
\end{equation}
where $m, k>0$ are parameters, and the smooth potential is bounded:
\begin{equation}\label{boundsForPotential}
  0\le U(q)\le U_{\sup}<+\infty
  \qquad\forall q\in \mathbb{R}^n.
\end{equation}
The same argument as in the previous section shows that we have global existence in the future. 

We cannot expect global existence in the past already in the simple one-di\-men\-sion\-al example $m\ddot q=- k\lvert \dot q\rvert\dot q$, $q\in \mathbb{R}$, for which all nonconstant solutions are of the form $q(t)= \frac{m}{k}\log(\omega(t-t_0))$, for parameters $\omega>0$, $t_0\in\mathbb{R}$, which are only defined for $t>t_0$.

To investigate possible non-globality in the past in the general case of hydraulic resistance in a bounded potential field, let us put this system into the Lagrange nonvariational formulation~\eqref{Lagrange} with the choices
\begin{equation}
  L(t,q,\dot q):=
  \frac{1}{2}m\lvert\dot q\rvert^2-U(q),\qquad
  Q(t,q,\dot q):=-k \dot q\lvert\dot q\rvert.
\end{equation}
If we take the family $q_\lambda(t):=q(t+\lambda e^{-at})$, with $a>0$, from formula~\eqref{veryGeneralConstantAlongMotion}, we obtain the following constant of motion:
\begin{equation}
  e^{-at}m\lvert\dot q(t)\rvert^2
  +\int_{t_0}^t e^{-as}\Bigl(
  \lvert\dot q(s)\rvert^3+am\lvert\dot q(s)\rvert^2
  +\nabla U\bigl(q(s)\bigr)\cdot\dot q(s)
  -m\ddot q(s)\cdot\dot q(s)\Bigr)ds,
\end{equation}
which can be rewritten, after a couple of integrations by parts, as
\begin{multline}\label{quadraticConstantOfMotion}
  \frac{m}{2}e^{-at}\lvert\dot q(t)\rvert^2
  +e^{-at}U\bigl(q(t)\bigr)-e^{-at_0}U\bigl(q(t_0)\bigr)+\\
  +\frac{1}{2}\int_{t_0}^t e^{-as}\Bigl(
  2k\lvert\dot q(s)\rvert^3
  +am\lvert\dot q(s)\rvert^2
  +2aU\bigl(q(s)\bigr)\Bigr)ds
  \equiv \frac{m}{2}e^{-at_0}\lvert\dot q(t_0)\rvert^2.
\end{multline}
Crucially, the left-hand side is monotonic with respect to the value of $\lvert\dot q\rvert$.

When $t<t_0$, from~\eqref{quadraticConstantOfMotion} and~\eqref{boundsForPotential} we can write the inequality
\begin{align}
  \frac{m}{2}e^{-at_0}\lvert\dot q(t_0)\rvert^2\equiv{}&
  \frac{m}{2}e^{-at}\lvert\dot q(t)\rvert^2
  +e^{-at}U\bigl(q(t)\bigr)-e^{-at_0}U\bigl(q(t_0)\bigr)+{}\notag\\
  &{}+\frac{1}{2}\int_{t_0}^t e^{-ks}\Bigl(
  2k\lvert\dot q(s)\rvert^3+am\lvert\dot q(s)\rvert^2
  +2aU\bigl(q(s)\bigr)\Bigr)ds\le\notag\\
  \le{}&\frac{m}{2}e^{-at}\lvert\dot q(t)\rvert^2
  +e^{-at}U_{\sup}+{}\\
  &{}+\frac{1}{2}\int_{t_0}^t e^{-as}\Bigl(
  2k\lvert\dot q(s)\rvert^3
  +am\lvert\dot q(s)\rvert^2\Bigr)ds.
  \label{quadraticInequality}
\end{align}
We wish to compare the smooth scalar function $t\mapsto\lvert\dot q(t)\rvert^2$ with the solution $z(t)$ of the integral equation
\begin{equation}\label{integralEquation}
 \frac{m}{2}e^{-at_0}\lvert\dot q(t_0)\rvert^2\equiv
  \frac{m}{2}e^{-at}z(t)
   +e^{-at}U_{\sup}+\frac{1}{2}\int_{t_0}^t e^{-as}\Bigl(
  2k\bigl|z(s)\bigr|^{3/2}+az(s)\Bigr)ds,
\end{equation}
which is equivalent to a Cauchy problem for a differential equation with separated variables:
\begin{gather}\label{auxiliaryDiffEq}
  \frac{z'(t)}{kz(t)^{3/2}-aU_{\sup}}=-\frac{2}{m},\\
  \label{initialValueForZ}
  z(t_0)=\lvert\dot q(t_0)\rvert^2-\frac{2}{m}U_{\sup}.
\end{gather}
Suppose that
\begin{equation}\label{initialCondition}
  \lvert\dot q(t_0)\rvert^2>
  \frac{2U_{\sup}}{m}+\Bigl(\frac{aU_{\sup}}{k}\Bigr)^{2/3}\,,
\end{equation}
so that $kz(t_0)^{3/2}-aU_{\sup}>0$. Then the denominator in~\eqref{auxiliaryDiffEq} is~$>0$, $z(t)$ is decreasing and it explodes in the past at a finite time~$t^*<t_0$ given by integrating the differential equation:
\begin{equation}\label{explosionTime}
  \int_{z(t_0)}^{+\infty}\frac{du}{ku^{3/2}-aU_{\sup}}=
  -\frac{2}{m}(t^*-t_0).
\end{equation}
The inequality $z(t)<\lvert\dot q(t)\rvert^2$ holds in a neighbourhood of $t=t_0$. To prove that it holds for all $t\in\mathopen\rbrack t^*,t_0\rbrack$, suppose that there exists a time $t_1<0$ such that $z(t_1)=\lvert\dot q(t_1)\rvert^2$ and that $z(t)<\lvert\dot q(t)\rvert^2$ holds for all $t\in\mathopen\rbrack t_1,t_0\rbrack$. Then we can concatenate 
\eqref{quadraticInequality} with~\eqref{integralEquation}:
\begin{align*}
  \frac{m}{2}e^{-at_0}{}&\lvert\dot q(t_0)\rvert^2 \le{}\\
  \le{}&\frac{m}{2}e^{-at_1}\lvert\dot q(t_1)\rvert^2
  +e^{-at_1}U_{\sup}+
  \frac{1}{2}\int_{t_0}^{t_1} e^{-as}\Bigl(
  2k\lvert\dot q(s)\rvert^3
  +a\lvert\dot q(s)\rvert^2\Bigr)ds=\\
  ={}&\frac{m}{2}e^{-at_1}z(t_1)
  +e^{-at_1}U_{\sup}+ \frac{1}{2}\int_{t_0}^{t_1} e^{-as}\Bigl(
  2k\lvert\dot q(s)\rvert^3+a\lvert\dot q(s)\rvert^2\Bigr)ds
  <\\
  <{}&\frac{m}{2}e^{-at_1}z(t_1)
  +e^{-at_1}U_{\sup}+ \frac{1}{2}\int_{t_0}^{t_1} e^{-as}\Bigl(
  2k\lvert z(s)\rvert^{3/2}+az(s)\Bigr)ds=\\
  ={}&\frac{m}{2}e^{-at_0}\lvert\dot q(t_0)\rvert^2,
\end{align*}
which is impossible. We conclude that, for $t<t_0$, $\lvert\dot q(t)\rvert^2$ is controlled from below, as long as it exists, by a function $z(t)$ that explodes to~$+\infty$ in finite time. Since the constant $a>0$ can be chosen arbitrarily small, the inequality~\eqref{initialCondition} can be replaced by $\lvert\dot q(t_0)\rvert^2>2U_{\sup}/m$, which is nicely equivalent to $\frac{m}{2}\lvert\dot q(t_0)\rvert^2>U_{\sup}$.

\emph{Conclusion: if $0\le U\le U_{\sup}<+\infty$, all the solutions to the differential equation~\eqref{quadraticResistanceEquation} for which the initial kinetic energy $\frac{m}{2}\lvert\dot q(t_0) \rvert^2$ is strictly greater than $U_{\sup}$ explode in the past in finite time}. 

\section{The Maxwell-Bloch equations}\label{MBequations}

The Maxwell-Bloch equations are well-known to describe laser dynamics for a system of two-level atoms in a cavity resonator. They were first derived in a 1965 paper by Arecchi et Bonifacio. The so called resonant case can be written as
\begin{equation}\label{ELDissipativeConcrete}
  \begin{cases}
  \ddot q_1=-abq_1-(a+b)\dot q_1+g^2q_1\dot q_3\\
  \ddot q_2=-abq_2-(a+b)\dot q_2+g^2q_2\dot q_3\\
  \ddot q_3=-a\bigl(q_1^2+q_2^2\bigr)
  -c(\dot q_3-k)
  -\bigl(q_1\dot q_1+q_2\dot q_2\bigr),
  \end{cases}
\end{equation} 
which has the Lagrangian form~\eqref{Lagrange} with the following choice of $L,Q$:
\begin{gather}\label{MB}
  L=\frac{1}{2}
  \bigl(\dot q_1^2+\dot q_2^2+g^2\dot q_3^2+(q_1^2+q_2^2)
  (g^2\dot q_3-ab\bigr)\bigr), \\
  \label{Qdissipative}
  Q=-\Bigl((a+b)\dot q_1,
  (a+b)\dot q_2,
  ag^2\bigl(q_1^2+q_2^2\bigr)+cg^2(\dot q_3-k)\Bigr),
\end{gather}
where $a,b,c\ge 0$, $g>0$, $k\in\mathbb{R}$ are parameters (see Arecchi and Meucci~\cite{ArecchiMeucci} and our paper with Residori~\cite{laserdiss}). We are going to briefly describe two kinds of nonstandard separation of variables that hold when $a=b=c=0$ (conservative case) and when $a,b,c>0$ (dissipative) with $c=2a$. From these separations, we deduced or conjectured some dynamical features which we will not repeat here. All details are in our paper~\cite{GZMBcons} in the conservative case, and in the already cited~\cite{laserdiss} in the dissipative case.

\subsection{Conservative case $a=b=c=0$}
\label{conservativeSubsection}

Using the anisotropic scaling family
\begin{equation}
  q_\lambda(t)=\bigl(e^{\lambda}q_1(t), 
  e^{\lambda}q_2(t), e^{-2\lambda}q_3(t)\bigr)\,,
\end{equation}
the constant of motion~\eqref{veryGeneralConstantAlongMotion} becomes
\begin{equation}\label{prefish}
  \ddot q_3(t)+2Bg^2 q_3(t)+2Et-3g^2\int_{t_0}^t \dot q_3(s)^2ds\,,
\end{equation}
were $B,E$ are the known first integrals
\begin{equation}
  E=\frac{1}{2}\bigl(\dot q_1^2+\dot q_2^2+g^2\dot q_3^2\bigr),
  \qquad
  B=\dot q_3+\frac{1}{2}(q_1^2+q_2^2).
 \end{equation}
By derivation of \eqref{prefish} with respect to~$t$ we get the differential equation of order~2 for~$z:=\dot q_3$ 
\begin{equation}\label{thirdorder}
  \ddot z(t)+2B g^2 z(t)+2E-3 g^2 z(t)^2=0.
\end{equation}
which has two equilibria with a \emph{fish}-shaped separatrix. From this well-known equation it is easy to classify the conditions for the solution in $z:=\dot q_3$ to be periodic or homoclinic, as shown in Fig.~\ref{genericAndHomoclinic}.

The $(q_1,q_2)$ obeys a central force dynamics. Indeed, plugging $\dot q_3=B-\frac{1}{2} r^2$, with $r^2=q_1^2+q_2^2$, into the first two Lagrange equations we have
\begin{equation}
  \ddot q_1=g^2\Bigl(B-\frac{1}{2} r^2\Bigr)q_1,\qquad
  \ddot q_2=g^2\Bigl(B-\frac{1}{2} r^2\Bigr)q_2.
\end{equation}
 
\begin{figure}
\centering
\includegraphics[width=0.45\textwidth]{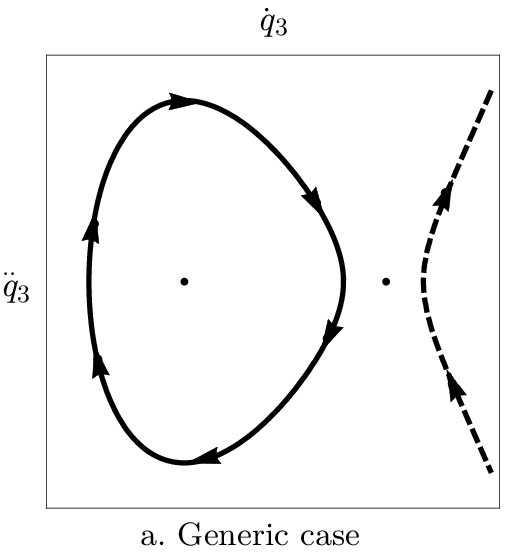}%
\qquad\quad
\includegraphics[width=0.45\textwidth]{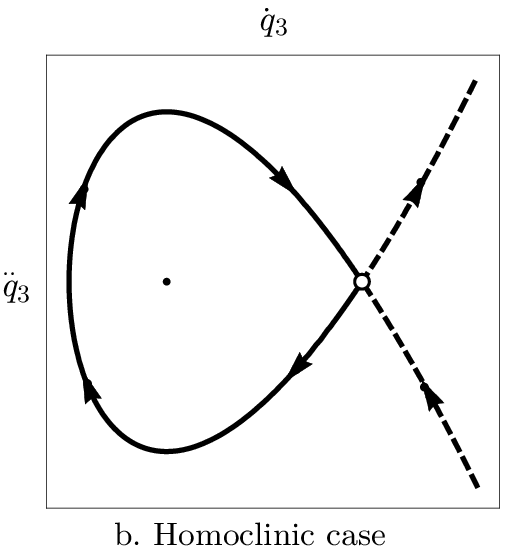}
\caption{Generic (periodic) and homoclinic orbits of $(\dot q_3,\ddot q_3)$ in the conservative case of the Maxwell-Bloch equations (Subsec.~\ref{conservativeSubsection}); the dashed lines and the dots are not visited by solutions, but they are level sets or stationary points of the potential function associated with equation~\eqref{thirdorder}.}
\label{genericAndHomoclinic}
\end{figure}


\subsection{Dissipative case $a,b,c>0$}
\label{dissipativeSubsection}

The family $q_\lambda(t)= q(t)+\lambda(0,0,2e^{ct})$ in~\eqref{veryGeneralConstantAlongMotion} gives the constant of motion
\begin{equation*}
 g^2e^{ct}\bigl(q_1(t)^2+q_2(t)^2+2\dot q_3(t)-2k\bigr)+
  (2a-c)g^2\int e^{ct}\bigl(q_1(t)^2+q_2(t)^2\bigr)dt.
\end{equation*}
When $c=2a$ we have a first integral 
\begin{equation}\label{firstIntegralM}
  M= e^{2at}\bigl(q_1^2+q_2^2+2\dot q_3-2k\bigr)\,,
\end{equation}
which permits a separation of the $q_1,q_2$ variables. In polar coordinates $(r,\theta)$ in the $q_1,q_2$ plane we have
\begin{gather}\label{separataR}
  \ddot r=-(a+b)\dot r+
  \Bigl(g^2k-ab+\frac{g^2M}{2} e^{-2at}\Bigr)r
  -\frac{g^2}{2}r^3+
  \frac{N^2}{r^3}e^{-2(a+b)t},\\
  \label{arealintegral}
  N=e^{(a+b)t}r^2\dot\theta, 
\end{gather}
where $N$ is another constant of motion which can be deduced from \eqref{veryGeneralConstantAlongMotion} with the rotation family
  \begin{equation}
  q_\lambda(t)=
  \bigl(q_1(t)\cos\lambda+
  q_2(t)\sin\lambda,
  -q_1(t)\sin\lambda+q_2(t)\cos\lambda,
  q_3(t)\bigr).
\end{equation}
We can solve for the remaining variable~$q_3$ using the first integral~$M$ of formula~\eqref{firstIntegralM}:
\begin{equation}
  \dot q_3(t)= k-\frac{1}{2}r(t)^2+\frac{1}{2}Me^{-2at}.
\end{equation}

Let us have a look at what happens if the time exponentials $e^{-2at}$ and $e^{-2(a+b)t}$ in equation~\eqref{separataR} are replaced by their limit~0 as $t\to+\infty$:
\begin{equation}\label{separataRasintotica}
  \ddot r=-(a+b)\dot r+(g^2k-ab)r
  -\frac{g^2}{2}r^3.
\end{equation}
This limiting equation has constant solutions corresponding to the nonnegative solutions of the algebraic equation
\begin{equation}\label{equilibriumRadiusEquation}
  (g^2k-ab)r-\frac{g^2}{2}r^3=0.
\end{equation}
There are clearly two cases:

\begin{itemize}

\item If $g^2k\le ab$ then equation~\eqref{equilibriumRadiusEquation} has only the solution $r=0$. We conjecture that $r(t)\to0$ for all solutions of the original equation~\eqref{separataR}. Figure~\ref{dissipativeBlochFigures}a shows such a trajectory on the $q_1,q_2$ plane.

\item If $g^2k>ab$ we have two nonnegative solutions $r=0$ and
\begin{equation}\label{asymtpticRadius}
  r_\infty=\frac{1}{g}\sqrt{2(g^2k-ab)}.
\end{equation}
The conjecture is that the nontrivial solutions in the $q_1,q_2$ plane converge to a point on the circle with radius $r_\infty$ and center in the origin, as in Figure~\ref{dissipativeBlochFigures}b.
\end{itemize}

\begin{figure}
\centering
\includegraphics[width=0.45\textwidth]{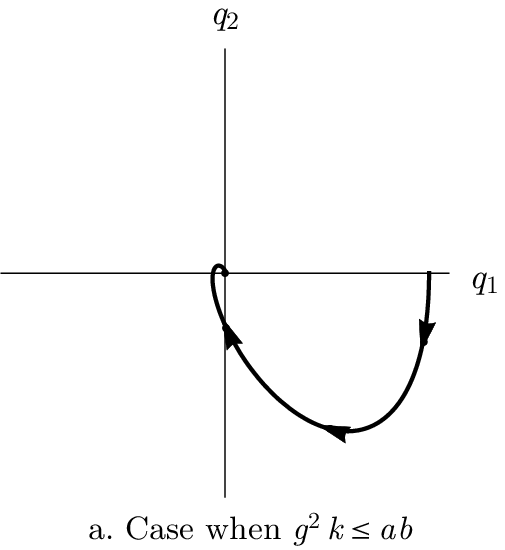}%
\qquad \quad
\includegraphics[width=0.45\textwidth]{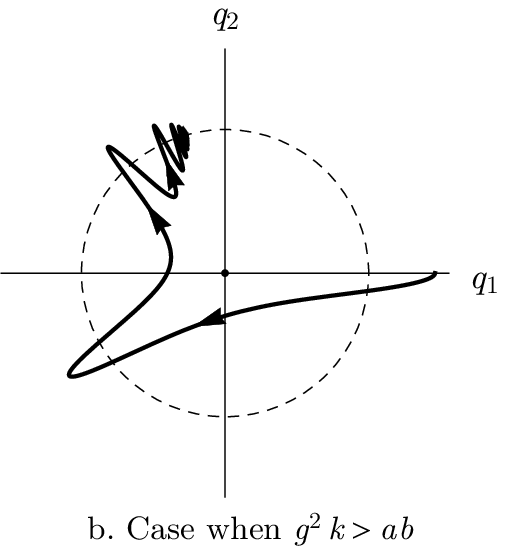}
\caption{Projection of forward orbits on the $q_1,q_2$ plane in two dissipative cases with $c=2a$ of the Maxwell-Bloch equations (Subsec.~\ref{dissipativeSubsection}), computed numerically. On the left with $g^2k\le ab$ the solution goes to the origin; on the right with $g^2k> ab$ the orbit converges to a point on the (dashed) circle with radius $r_\infty$, as in equation~\eqref{asymtpticRadius}.}
\label{dissipativeBlochFigures}
\end{figure}


\section{Conclusions}

When studying mechanical systems with a Lagrangian structure, we think that it is worthwhile to apply Theorem~\ref{mainTheorem} in search of useful integral constants of motion. At the moment the choice of the family~$q_\lambda$ is more of an art, rather than a science. However, we hope this paper provides enough concrete examples to stimulate the curiosity of the reader.


\section*{Acknowledgments}
 The research was done under the auspices of INdAM (Isti\-tuto Nazio\-nale di Alta Mate\-ma\-tica). G.Z. is deeply grateful to his surgeon prof. Fede\-ri\-co Rea, a true luminary.



\begin{thebibliography}{9}
\frenchspacing

\bibitem{ArecchiMeucci}
\newblock  F. T. Arecchi and R. Meucci,
\newblock  Chaos in lasers.
\newblock  \emph{Scholarpedia} 3(9):7066 (2008).

\bibitem{Calogero}
\newblock F. Calogero,
\newblock  Solutions of the  one dimensional $n$-body problems with quadratic and/or inversely quadratic pair potentials.
\newblock  \emph{J. Math. Phys}, \textbf{12} (1971),
419--436.

\bibitem{GorniZampieri}
\newblock  G. Gorni and G. Zampieri,
\newblock  Revisiting Noether's theorem on constants of motion.
\newblock  \emph{Journal of Nonlinear Mathematical Physics}, \textbf{21},
No.~1 (2014), 43--73.

\bibitem{GZviscous}
\newblock  G. Gorni and G. Zampieri,
\newblock  Nonlocal variational constants of motion in dissipative dynamics.
\newblock \emph{Differential and Integral Equations}, \textbf{30} (2017), 631--640. 

\bibitem{GZMBcons}
\newblock  G. Gorni and G. Zampieri,
\newblock  Nonstandard separation of variables for the Maxwell-Bloch conservative system.
\newblock \emph{S\~ao Paulo J. Math. Sci.}, \textbf{12}, No.~1  (2018), 146--169.

\bibitem{GZKilling}
\newblock  G. Gorni and G. Zampieri,
\newblock  Nonlocal and nonvariational
 extensions of Killing- type equations.
\newblock \emph{Discrete Contin. Dyn. Syst. Ser. S}, \textbf{11}, No.~4 (2018), 675--689. 

\bibitem{laserdiss} 
\newblock G.  Gorni, S. Residori, G. Zampieri,
\newblock A quasi separable 
dissipative Maxwell- Bloch  system for laser dynamics. 
\newblock \emph{Qual. Theory Dyn. Syst.}, \textbf{18}, No.~2 (2019),  371--381.   

 
\bibitem{Leach}
\newblock  P. G. L. Leach,
\newblock  Lie symmetries and Noether symmetries.
\newblock  \emph{Applicable Analysis and Discrete Mathematics}, \textbf{6} (2012),
238--246.

\bibitem{Leone}
\newblock  R. Leone and T. Gourieux,
\newblock  Classical Noether theory with application to the linearly damped particle.
\newblock  \emph{European Journal of Physics}, \textbf{36} (2015)
065022, 20pp.

\end{thebibliography}
\end{document}